\documentclass[11pt]{article}

\usepackage{graphicx}
\usepackage{latexsym,amssymb}
\usepackage{amsthm}
\usepackage{indentfirst}
\usepackage{amsmath}
\usepackage{color}
\usepackage{xcolor}
\usepackage{hyperref}
\usepackage{ulem}
\usepackage{float}
\usepackage{textcomp}

\newtheorem{theoremalph}{Theorem}

\newtheorem{Theorem}{Theorem}[section]
\newtheorem*{maintheorem}{Main Theorem}
\newtheorem*{Theorem A}{Theorem A}
\newtheorem*{Theorem B}{Theorem B}

\newtheorem{Definition}[Theorem]{Definition}

\newtheorem{Proposition}[Theorem]{Proposition}

\newtheorem{Lemma}[Theorem]{Lemma}

\newtheorem{Remark}[Theorem]{Remark}

\newtheorem*{Acknowledgements}{Acknowledgements}

\DeclareMathOperator{\diam}{diam}

\DeclareMathOperator{\vol}{Vol}

\newcommand\vl{\operatorname{\underline{v}}}

\begin{document}

\newpage
\title{Entropies and volume growth of unstable manifolds}
\author{ Yuntao Zang\footnote{Yuntao Zang is supported by China Scholarship Council (CSC)}}

\maketitle
\begin{abstract}
Let $f$ be a $C^2$ diffeomorphism on a compact manifold. Ledrappier and Young introduced entropies along unstable foliations for an ergodic measure $\mu$. We relate those entropies to covering numbers in order to give a new upper bound on the metric entropy of $\mu$ in terms of Lyapunov exponents and topological entropy or volume growth of sub-manifolds. We also discuss extensions to the $C^{1+\alpha},\,\alpha>0$ case.
	
\end{abstract}

\section{Introduction}
\newcommand\loc{{\operatorname{loc}}}
\newcommand\essinf{\operatorname{essinf}}
\newcommand\Diff{\operatorname{Diff}}
\newcommand\Proberg{P_{\operatorname{erg}}}
\renewcommand\top{{\rm top}}

Entropy is a fundamental invariant in dynamics. It can be defined in the topological, ergodic or differentiable categories and quantifies the dynamical complexity. The classical result on the connection between entropy and Lyapunov exponents is the Margulis-Ruelle Inequality  \cite{Rue78}.  It states that for a $C^1$ map $f$ on a compact manifold $M$ and an ergodic invariant Borel probability measure $\mu$, 
\begin{align}
h(f,\mu)\leq\sum_{i=1}^{u}\lambda_{i}(f,\mu)\cdot\dim E^{i}\label{Ruelle-Margulis inequality}
\end{align} where $\lambda_{i}(f,\mu),1\leq i\leq u$ are the positive Lyapunov exponents and $E^{i},1\leq i\leq u$ are the corresponding Oseledets' vector bundles.  
As perhaps first observed by Katok, this inequality implies that measures with positive entropy of surface diffeomorphisms are hyperbolic, i.e., without zero Lyapunov exponents. Katok \cite{Kat80} was then able to analyze such dynamics using the Pesin theory in the $C^{1+\alpha}$ setting.

Also using Pesin theory, Newhouse \cite{New88} proved another bound for the entropy of an ergodic measure  this time by the volume growth of sub-manifolds which are transverse to its stable manifolds. In the $C^{1}$ setting with dominated splitting, without using Pesin theory, Saghin \cite{Sag14} and Guo-Liao-Sun-Yang \cite{GLS18} bounded above the metric entropy  by a mixture between the positive Lyapunov exponents and the volume growth of some sub-manifold. By using Ledrappier-Young's result \cite{LeY85}, Cogswell \cite{Cog00} proved that the volume growth of local unstable manifolds is larger than the metric entropy. Cogswell's proof assumes  $C^{2}$ smoothness since this is required in Ledrappier-Young's work. On the topological side, for $C^{1+\alpha}$ diffeomorphism, Przytycki \cite{Prz80} proved that the topological entropy is bounded above by the growth rate of some differential forms. Later Kozlovski \cite{Koz98} showed it is an equality if the system is $C^{\infty}$.

In this paper, we generalize Cogswell's idea from  \cite{Cog00} to establish a more general upper bound for  $C^{2}$ systems without assuming dominated splitting. We bound the entropy of a measure by a combination of  Lyapunov exponents (as in Ruelle's inequality) and various growths of unstable manifolds such as volume growth (as in Newhouse's inequality \cite{New88}). 
In a forthcoming work, we will use this new bound to extend the previously mentioned Katok's hyperbolicity argument beyond dimension two.

Our proof is a combination of Ledrappier-Young's entropy formula \cite{LeY85} and Pesin theory. We also discuss some extensions for hyperbolic measures in $C^{1+\alpha}$ case.

\medbreak

\begin{maintheorem}\label{maintheorem}
Let $f$ be a $C^{2}$ diffeomorphism on a compact manifold $M$ and let $\mu$ be any ergodic, invariant probability measure. Consider its positive Lyapunov exponents $\lambda_1>\dots>\lambda_u$ and the corresponding $i$-th local unstable manifolds $W^i_{\rm loc}(x)$ for almost every $x\in M$ and $i=1,\dots,u$.

Then the entropy $h(f,\mu)$ is bounded, for any index $1\leq i\leq u$, by the sum of the almost everywhere volume growth of $W^i_{\rm loc}(x)$ and the transverse Lyapunov exponents $\lambda_{i+1},\dots,\lambda_u$, repeated according to multiplicity. 

In this inequality, the volume growth can be replaced by fibered entropy or topological entropy of $W^i_{\rm loc}(x)$.
\end{maintheorem}

We give complete and precise statements in the next section after introducing the required notions. See in particular   Theorem \ref{upper bound mixture C2 setting}.

\subsection{Definitions}

Let $f$ be a $C^{1+\alpha}$ ($\alpha>0$) diffeomorphism on a compact manifold $M$, that is, $f$ is differentiable and its differential is H\"older-continuous with some positive exponent $\alpha$. Let $\mu$ be an ergodic probability measure. Oseledets' Theorem\cite{Ose68} states that there are an invariant measurable subset $R_{\mu}$ with full measure, an invariant measurable decomposition $T_{R_{\mu}}M=E^{1}\oplus E^{2}\oplus\cdots\oplus E^{l}$ and finitely many numbers $\lambda_{1}>\lambda_{2}>\cdots>\lambda_{l}$ such that for any $x\in R_{\mu}$ and any nonzero vector $v\in E^{j}_{x}$, we have $$\lim_{n\to \pm\infty}\frac{1}{n}\log||Df^{n}_{x}(v)||=\lambda_{j}.$$

We list the positive Lyapunov exponents as $\lambda_{1}>\lambda_{2}>\cdots>\lambda_{u}$. By Pesin theory, for $1\leq i\leq u$ and for any $x\in R_{\mu}$, the \emph{$i$-th global unstable manifold} $$W^{i}(x)\triangleq\{y\in M\big|\, \limsup_{n\to+\infty}\frac{1}{n}\log d(f^{-n}(x),f^{-n}(y))\leq-\lambda_{i}\}$$ is a $C^{1+\alpha}$ immersed sub-manifold. 

We define the \emph{$i$-th local unstable manifold}, $$W^{i}_{\rho}(x)\triangleq \,\text{ connected part of}\,\,W^{i}(x)\cap B(x,\rho)\text{ containing }x$$ where $B(x,\rho)$ is the ball centered at $x$ with radius $\rho$. 
At each $x\in R_{\mu}$, we fix a positive number $r(x)$ such that 
$W^{i}_{r(x)}(x)$ is an embedded sub-manifold.

\begin{Definition}\label{existence of subordinate partition}
	Let $f$ be a $C^{1+\alpha}$ diffeomorphism on a compact manifold $M$ and let $\mu$ be an invariant measure. For $1\leq i\leq u$, we say a measurable partition $\xi^{i}$ is \emph{subordinate to $W^{i}$} if for $\mu$-a.e. $x$, 
	\begin{itemize}
		\item $\xi^{i}(x)\subset W^{i}(x)$,
		\item $\xi^{i}(x)$ contains an open neighborhood of $x$ w.r.t. the intrinsic topology on $W^{i}(x)$.
	\end{itemize}
	
\end{Definition}
\begin{Remark}$ $
	\begin{itemize}
		\item We refer to  \cite{Roh52} for background on measurable partitions and associated systems of conditional measures.
		\item 	
		Lemma 9.1.1 in  \cite{LeY85} shows the existence of increasing subordinate measurable partitions. Here a partition $\eta$ is called \emph{increasing} if $\eta(x)\subset f(\eta(f^{-1}(x)))$ for $\mu$-a.e. $x$.
	\end{itemize}	
\end{Remark}

From now on, we fix a family of measurable partitions $\{\xi^{i}\}_{1\leq i\leq u}$ subordinate to $\{W^{i}\}_{1\leq i\leq u}$.
For $1\leq i\leq u$, let $\{\mu_{x}^{i}\}$ be the family of conditional measures w.r.t. the measurable partition $\xi^{i}$. Ledrappier and Young \cite{LeY85} have defined the \emph{entropy along $i$-unstable foliation} $h_{i}(f,\mu)$(for more detail, see Proposition \ref{well defined hi}) by a fibered version of Brin-Katok's formula, namely: $$h_{i}(f,\mu)\triangleq\lim_{\tau\to 0}\liminf_{n\to+\infty}-\frac{1}{n}\log\mu_{x}^{i}(V^{i}(x,n,\tau))$$ where $$V^{i}(x,n,\tau)\triangleq\{y\in W^{i}_{r(x)}(x)|\,d(f^{j}(x),f^{j}(y))\leq\tau,\,0\leq j\leq n-1\}.$$ 
Remark that here in the definition of the dynamical ball $V^{i}(x,n,\tau)$, we use the global metric $d$ on $M$, unlike the definition in  \cite{LeY85} that uses the intrinsic metric on the sub-manifold $W^{i}(x)$. But since we only consider the case when $\tau \to 0$, our definition of $h_{i}(f,\mu)$ coincides with theirs.

The volume of a sub-manifold $\gamma\subset M$ of constant dimension is denoted by $\vol(\gamma)$. The lower volume growth of such a sub-manifold $\gamma\subset M$ with $\vol(\gamma)<\infty$ is:
$$
\vl(f,\gamma)\triangleq\liminf_{n\to\infty} \frac1n \log^+ \vol(f^{n}(\gamma))
$$ where $\log^{+} a=\max\{0,\log a\}$.

We now introduce the key concepts of our results. They are well defined by Lemma \ref{four functions are well defined} in Section \ref{basic properties}.
\begin{Definition}
Given $1\leq i\leq u$, the \emph{$\mu$-a.e. lower volume growth rate of $W^{i}$} is the $\mu$-a.e. value of $$\vl_{i}(f,\mu)\triangleq \inf_{\rho}\vl(f,W^{i}_{\rho}(x)).$$
\end{Definition}
	Let $E(n,\varepsilon,\gamma)$ denote a maximal $(n,\varepsilon)$ separated subset of a $C^{1}$ sub-manifold $\gamma$. For the definitions of separated subset and some other basic concepts in ergodic theory, see the book \cite{Wal82}.

\begin{Definition}
	Given $1\leq i\leq u$, the \emph{$\mu$-a.e. lower topological entropy of $W^{i}$} is the $\mu$-a.e. value of $$\underline{h}^{i}_{\top}(f,\mu)\triangleq \inf_{\rho}\lim_{\varepsilon\to 0}\liminf_{n\to+\infty}\frac{1}{n}\log \#E(n,\varepsilon,W^{i}_{\rho}(x)).$$
\end{Definition}
\begin{Remark}
	Recall that the topological entropy of $W^{i}_{\rho}(x)$ is $$h_{\top}(f,W^{i}_{\rho}(x))\triangleq \lim_{\varepsilon\to 0}\limsup_{n\to+\infty}\frac{1}{n}\log \#E(n,\varepsilon,W^{i}_{\rho}(x)).$$ Hence we have $\underline{h}^{i}_{\top}(f,\mu)\leq \inf_{\rho}h_{\top}(f,W^{i}_{\rho}(x))$ for any $x\in R_{\mu}$. Note that  $\inf_{\rho}h_{\top}(f,W^{i}_{\rho}(x))$ is also $\mu$-a.e. constant.
\end{Remark}

For $1\leq i\leq u,\, x\in R_{\mu}$ and $\lambda>0$, define $$N_{\lambda}(\mu^{i}_{x},n,\varepsilon)\triangleq\min\{\#C\subset R_{\mu}:\mu^{i}_{x}(\bigcup_{x\in C}V^{i}(x,n,\varepsilon))\geq\lambda\}.$$

\begin{Definition}
	Given $1\leq i\leq u$, the \emph{upper fibered Katok entropy  of $W^{i}$} is the $\mu$-a.e. value of $$\overline{h}_{i}^{K}(f,\mu)\triangleq \inf_{\lambda}\lim_{\varepsilon\to 0}\limsup_{n\to+\infty}\frac{1}{n}\log N_{\lambda}(\mu^{i}_{x},n,\varepsilon).$$
	Similarly, the \emph{lower fibered Katok entropy of $W^{i}$} is the $\mu$-a.e. value of $$\underline{h}_{i}^{K}(f,\mu)\triangleq \inf_{\lambda}\lim_{\varepsilon\to 0}\liminf_{n\to+\infty}\frac{1}{n}\log N_{\lambda}(\mu^{i}_{x},n,\varepsilon).$$
\end{Definition}
\begin{Remark}
	The above definition is analogous to the formula of  Katok in  \cite{Kat80}, expressing the metric entropy as the growth rate of the cardinality of maximal separated sets. 

\end{Remark}

~\\
\subsection{Main results}
From now on, when we say $C^{1+\alpha}$ diffeomorphism, we always assume $\alpha>0$.
\begin{theoremalph}\label{main result}
	Let $f$ be a $C^{1+\alpha}$ diffeomorphism on a compact manifold $M$. Let $\mu$ be an ergodic measure.  List the positive Lyapunov exponents of $\mu$ as $\lambda_{1}>\lambda_{2}>\cdots>\lambda_{u}>0$. Then for $1\leq i\leq u$, the entropy along the $i$-th unstable foliation satisfies:
	 \begin{enumerate}
	  \item $h_{i}(f,\mu)=\underline{h}_{i}^{K}(f,\mu)=\overline{h}_{i}^{K}(f,\mu)$;
	  \item $h_i(f,\mu)\leq \underline{h}^{i}_{\top}(f,\mu)$;
	  \item $h_i(f,\mu)\leq \vl_{i}(f,\mu)$.
	 \end{enumerate}
\end{theoremalph}

Let $h(f,\mu)$ be the entropy of $\mu$. When $f$ is $C^{2}$, Ledrappier and Young have shown the following entropy formula (Theorem C$'$ in  \cite{LeY85}): for any $1\leq i\leq u$,
$$
h(f,\mu)= h_{i}(f,\mu)+\sum_{j=i+1}^u \lambda_j\cdot\gamma_j.
$$
where $\gamma_1,\dots,\gamma_u$ are some transverse dimensions satisfying $\gamma_j\leq \dim E^j$.
Therefore, Theorem~\ref{main result} immediately implies:
\begin{theoremalph}\label{upper bound mixture C2 setting}
	Let $f$ be a $C^{2}$ diffeomorphism on a compact manifold $M$. Let $\mu$ be an ergodic measure. List the positive Lyapunov exponents of $\mu$ as $\lambda_{1}>\lambda_{2}>\cdots>\lambda_{u}$. Then for $1\leq i\leq u$, $$h(f,\mu)=\underline{h}_{i}^{K}(f,\mu)+ \sum_{j=i+1}^u \lambda_j\cdot\gamma_j,$$ $$h(f,\mu)\leq \underline{h}^{i}_{\top}(f,\mu)+ \sum_{j=i+1}^u \lambda_j\cdot\gamma_j,$$ $$h(f,\mu)\leq\vl_{i}(f,\mu)+ \sum_{j=i+1}^u %\cdot
	\lambda_j\cdot\gamma_j.$$
\end{theoremalph}

The above contains the Main Theorem from page~\pageref{maintheorem}.

When the measure $\mu$ is hyperbolic, i.e., when $\mu$ has no zero Lyapunov exponents, the result in Theorem \ref{upper bound mixture C2 setting} is true for $i=u$ without the $C^{2}$ assumption. 

\begin{theoremalph}\label{upper bound mixture C1+ setting}
	Let $f$ be a $C^{1+\alpha}$ diffeomorphism on a compact manifold $M$. Let $\mu$ be an ergodic measure. If $\mu$ is hyperbolic, then $$h_{u}(f,\mu)=h(f,\mu).$$
\end{theoremalph}

\begin{Remark}
 As a consequence of Theorems \ref{main result}~and~\ref{upper bound mixture C1+ setting}, 
 $$h(f,\mu)=\underline{h}_{u}^{K}(f,\mu)=\overline{h}_{u}^{K}(f,\mu).$$ 
 Moreover, this quantity  is bounded above both by $\underline{h}^{u}_{\top}(f,\mu)$ and $\vl_{u}(f,\mu)$.
\end{Remark}

\subsection{Remarks}

Let us explain our motivation beyond the desire to prove natural inequalities.

Theorem \ref{upper bound mixture C2 setting} will be used in a forthcoming work to study some entropy-hyperbolic diffeomorphisms (as suggested by Buzzi \cite{Buz09}). More precisely, we will find a non-empty $C^{\infty}$ open set of diffeomorphisms which are not uniformly hyperbolic but whose ergodic measures of entropy close to the topological entropy are nevertheless hyperbolic and of given index. 

Theorem~\ref{upper bound mixture C1+ setting} extends by a simple argument the Ledrappier-Young entropy formula in the $C^{1+\alpha}$ setting assuming hyperbolicity. This  is used in some ongoing work by other authors (J.~Buzzi, S.~Crovisier, O.~Sarig). 

Note that A. Brown  \cite{Bro16} gives  this $C^{1+\alpha}$ generalization without the hyperbolicity assumption. More precisely, he gives a proof of a uniform bi-Lipschitz property of the stable holonomies inside center-unstable manifolds. However, his argument is technical and only a preprint at the time we are writing this. Hence we believe that our simple, half-page argument has some interest.

\section{Basic properties}\label{basic properties}
In this section, we list some basic results that will be used later.

\begin{Lemma}\label{four functions are well defined}
	Let $f$ be a $C^{1+\alpha}$ diffeomorphism on a compact manifold $M$. Let $\mu$ be an ergodic measure. Then the following four functions are constant almost everywhere.
	$$\inf_{\rho}\vl(f,W^{i}_{\rho}(x)),$$
	$$\inf_{\rho}\lim_{\varepsilon\to 0}\liminf_{n\to+\infty}\frac{1}{n}\log \#E(n,\varepsilon,W^{i}_{\rho}(x)),$$
	$$\inf_{\lambda}\lim_{\varepsilon\to 0}\limsup_{n\to+\infty}\frac{1}{n}\log N_{\lambda}(\mu^{i}_{x},n,\varepsilon),$$
	$$\inf_{\lambda}\lim_{\varepsilon\to 0}\liminf_{n\to+\infty}\frac{1}{n}\log N_{\lambda}(\mu^{i}_{x},n,\varepsilon).$$
\end{Lemma}
\begin{proof}
		By defintion, one can check that these functions are $f$-invariant. Hence by ergodicity, they are constant almost everywhere.
\end{proof}

Recall that $r(x)>0,\,x\in R_{\mu}$ is such that 
$W^{i}_{r(x)}(x)$ is an embedded sub-manifold. Indeed, by Pesin theory, we can assume for any $x\in R_{\mu}$, $r(x)$ is such that $$\lim_{n\to+\infty}\frac{1}{n}\log r(f^{n}(x))=0.$$
In light of this, we introduce in the following a collection of results in classical Pesin theory. For more detail, see section 8 in  \cite{LeY85} and Proposition 3.3 in  \cite{LeS82}.

\begin{Lemma}\label{pesin theory}
	Let $f$ be a $C^{1+\alpha}$ diffeomorphism on a compact manifold $M$ and let $\mu$ be an ergodic measure. List the positive Lyapunov exponents of $\mu$ as $\lambda_{1}>\lambda_{2}>\cdots>\lambda_{u}$. For any $\varepsilon>0$, we can find an increasing sequence of measurable sets $\Lambda^{\varepsilon}_{1}\subset \Lambda^{\varepsilon}_{2}\subset \cdots
	\subset\Lambda^{\varepsilon}_{k}\cdots\subset R_{\mu}$ and a sequence of numbers $\{r_{k}\}_{k\geq 1}$ with $0<r_{k}<1$ and $r_{k}\to 0$ such that 
	\begin{itemize}
		\item  $\cup_{k}\Lambda^{\varepsilon}_{k}=R_{\mu}$.
		\item $f^{n}(\Lambda^{\varepsilon}_{k})\subset\Lambda^{\varepsilon}_{k+n},\,\forall \,k,\,n\geq 1.$
		\item For any $x\in\Lambda^{\varepsilon}_{k}$, $r_{k}\leq r(x)$ and any $y\in W^{i}_{r_{k}}(x), 1\leq i\leq u$, $$d(f^{-n}(x),f^{-n}(y))\leq r_{k}^{-1}e^{-n(\lambda_{i}-\varepsilon)}d(x,y),\quad \forall\,n\geq 0.$$
		\item There is a constant $K$ such that for $k\geq 1$, $x\in\Lambda^{\varepsilon}_{k},\,\rho\leq r_{k}$ and $1\leq i\leq u$, $$\vol(W^{i}_{\rho}(x))\leq K\cdot \rho^{\sum_{j=1}^{i}\dim E^{j}}.$$
		\item $e^{-\varepsilon}\leq r_{k+1}/r_{k}\leq e^{\varepsilon}$,\,$\forall\,k\geq 1$.

	\end{itemize}

\end{Lemma}

\begin{Remark}
	\begin{itemize}$ $
		\item Here, for example, one can choose $r_k=e^{-\varepsilon k}$.
		\item Note that $W^{i}(x)$ is tangent to $\sum_{l=1}^{i}\dim E^{l}_{x}$ at $x$. These small numbers $\{r_{k}\}_{k\geq 1}$ indicate the size of Pesin charts. When $W^{i}_{\rho}(x)$ is in the Pesin chart of $x$, we can assume it is contained in a small cone around $x$ and therefore its volume is determined by its radius up to a uniform constant $K$.
	\end{itemize}
\end{Remark}

A standard Pesin Theory (e.g. remarks below Lemma 2.2.3 in  \cite{LeY85I}) shows that:
\begin{Lemma}\label{hyperbolicity implies points always in chart are in unstable manifold}
	Let $f$ be a $C^{1+\alpha}$ diffeomorphism on a compact manifold $M$ and let $\mu$ be an hyperbolic ergodic measure. Given $\varepsilon>0$ and $x\in R_{\mu}$, assume $x\in\Lambda^{\varepsilon}_{k}$ for some $k$. 
	Then 
	$$S^{cu}(x)\subset W^{u}(x)\quad \text{where}\quad S^{cu}(x)\triangleq\{y\in M\,\big|\,d(f^{-n}(x),f^{-n}(y))\leq r_{k}e^{-n\varepsilon},\,\forall\,n\geq 0\}.$$
\end{Lemma}
\begin{Remark}
	Roughly speaking, $S^{cu}(x)$ above is just the set of points whose backward trajectory always stays in the same Pesin chart of the backward trajectory of $x$. Hence in general, $S^{cu}(x)$ is the local center unstable manifold of $x$. But when the measure is hyperbolic, the above lemma says that $S^{cu}(x)$ reduces to the local unstable manifold.

\end{Remark}

For two measurable partitions $\xi$ and $\eta$, $\xi\vee\eta$ denotes the partition $\{\xi(x)\cap\eta(x)\}_{x\in R_{\mu}}$ and $\xi^{+}=\vee_{n=0}^{+\infty}f^{n}\xi$. Let $h_{\mu}(f,\xi)$ denote the entropy of $\xi$ w.r.t. $f$  and let $H_{\mu}(\xi|\eta)$ denote the mean conditional entropy.

The following result of Ledrappier and Young justifies the definition of the entropies along unstable foliations.

\begin{Proposition}[Proposition 7.2.1 in  \cite{LeY85}]\label{well defined hi}
	Let $f$ be a $C^{1+\alpha}$ diffeomorphism on a compact manifold $M$ and let $\mu$ be an ergodic measure. For $1\leq i\leq u$ and for any increasing partition $\xi^{i}$ subordinate to $W^{i}$  and for $\mu$ almost every point $x$, $$h_{i}(f,\mu)\triangleq\lim_{\tau\to 0}\liminf_{n\to+\infty}-\frac{1}{n}\log\mu^{i}_{x}(V^{i}(x,n,\tau))=\lim_{\tau\to 0}\limsup_{n\to+\infty}-\frac{1}{n}\log\mu^{i}_{x}(V^{i}(x,n,\tau))=H_{\mu}(\xi^{i}|f\xi^{i}).$$
\end{Proposition}
\begin{Remark}$ $ \label{entropy for increasing partition}
	\begin{itemize}
		\item Note that the functions of $x$ that appears in Proposition \ref{well defined hi} are $f$-invariant and therefore constant $\mu$ almost everywhere. They do not depend on the choice of the subordinate partition $\xi^{i}$. So it is proper to denote them by $h_{i}(f,\mu)$. See Lemma 3.12 in  \cite{LeY85I} for more detail.
		\item Ledrappier and Young \cite{LeY85} assume $C^{2}$ smoothness. But their proof of Proposition \ref{well defined hi} in their section 9 only uses Pesin Theory and $C^{1+\alpha}$ smoothness.
		\item Since $\xi^{i}$ is increasing, i.e., $\xi^{i}(x)\subset f(\xi^{i}(f^{-1}(x)))$ for $\mu$-a.e. $x$, we always have $$h_{\mu}(f,\xi^{i})\triangleq H_{\mu}(\xi|f(\xi^{+}))=H_{\mu}(\xi^{i}|f\xi^{i}).$$
\end{itemize}

\end{Remark}

We say $\eta$ is finer than $\xi$ denoted by $\xi\leq\eta$ if $\eta(x)\subset \xi(x)$ for $\mu$-a.e. $x$. For partitions with finite mean entropy, the finer partition has larger entropy. The following is an extension of this property to non-finite partitions.

\begin{Lemma}[Property 8.7 in  \cite{Roh67}]\label{finer partition has larger entropy}
	Let $f$ be a homeomorphism on a compact metric space $X$. Assume $\mu$ is an $f$-invariant probability measure.
	Let $\xi,\eta$ be two measurable partitions(possibly with infinite mean entropy) with $\eta$ being finer than $\xi$.
	If the mean conditional entropy $H_{\mu}(\eta\,|\,f(\xi^{+}))$ is finite, then $h_{\mu}(f,\xi)\leq h_{\mu}(f,\eta).$
\end{Lemma}

\begin{Remark}
Rokhlin's article \cite{Roh67} mainly discusses entropy theory for endormorphisms where most definitions and properties are stated by using $f^{-1}$. Since here we assume $f$ is a homeomorphism, our statement is parallel to the original statement of Property 8.7 in  \cite{Roh67}.
\end{Remark}
\begin{Lemma}\label{property of increasing partition}
	Let $f$ be a homeomorphism on a compact metric space $X$. Assume $\mu$ is an $f$-invariant probability measure.
	Let $\xi,\eta$ be two increasing measurable partitions with $h_{\mu}(f,\xi)<+\infty,h_{\mu}(f,\eta)<+\infty$. Then for any integer $n\geq 1$,$$h_{\mu}(f,\xi\vee\eta)=h_{\mu}(f,\xi\vee f^{n}\eta).$$
\end{Lemma}
\begin{proof}
	Since  $\xi,\eta$ are increasing, for any $n\geq 1$, we have $$f^{n}\xi\vee f^{n}\eta\leq\xi\vee f^{n}\eta\leq\xi\vee\eta.$$

In order to apply Lemma \ref{finer partition has larger entropy}, we note that

$$\begin{aligned}
H_{\mu}(\xi\vee \eta\,|\,f((\xi\vee f^{n}\eta)^{+}))
&= H_{\mu}(\xi\vee \eta\,|\,f\xi\vee f^{n+1}\eta)\\
&= H_{\mu}(\xi\,|\,f\xi\vee f^{n+1}\eta)+ H_{\mu}( \eta\,|\,\xi\vee f^{n+1}\eta)\\
&\leq H_{\mu}(\xi\,|\,f\xi)+ H_{\mu}(\eta\,|\,f^{n+1}\eta)\\
&=h_{\mu}(f,\xi)+n h_{\mu}(f,\eta)\\
&<+\infty.
\end{aligned}$$

By Lemma \ref{finer partition has larger entropy}, we have $$h_{\mu}(f,\xi\vee f^{n}\eta)\leq h_{\mu}(f,\xi\vee\eta).$$

To conclude, we prove the converse inequality by applying the previous one to $\xi_{1}=f^{n}\eta$ and $\eta_{1}=\xi$, obtaining:
 $$
    h_{\mu}(f,\xi\vee\eta) = h_{\mu}(f,f^{n}\xi\vee f^{n}\eta)\leq h_{\mu}(f,\xi\vee f^{n}\eta).
 $$
\end{proof}

The following is an extension of Lemma 3.1.2 in \cite{LeY85I}. The main difference is that here we only assume one of the two partitions is subordinate. The proof is essentially identical to Lemma 3.1.2 in \cite{LeY85I}. For completeness, we present it.
\begin{Lemma}\label{extension of the relations between subordinate partitions}
	Let $f$ be a $C^{1+\alpha}$ diffeomorphism on a compact manifold $M$ and let $\mu$ be an ergodic measure. Let $\xi^{u}$ be an increasing partition subordinate to $W^{u}$. Assume $\beta$ is a measurable partition satisfying 
	\begin{enumerate}
		\item $\beta$ is increasing,
		\item for $\mu$-a.e. $x$, $\beta(x)\subset W^{u}(x),$
		\item for $\mu$-a.e. $x$, $\diam((f^{-n}(\beta))(x))\to 0.$
	\end{enumerate} 
		Then $$h_{\mu}(f,\xi^{u}\vee\beta)=h_{\mu}(f,\beta).$$ 
\end{Lemma}
\begin{proof}
		Since both $\xi^{u}$ and $\beta$ are increasing and their entropies w.r.t. $f$ and $\mu$ are finite, by Lemma \ref{property of increasing partition}, for any $n\geq 1$, we have 
	$$\begin{aligned}
	h_{\mu}(f,\xi^{u}\vee\beta)
	&= h_{\mu}(f,(f^{n}\xi^{u})\vee\beta)\\
	&= H_{\mu}((f^{n}\xi^{u})\vee\beta\,|\,(f^{n+1}\xi^{u})\vee f\beta)\\
	&= H_{\mu}(\beta\,|\,(f^{n+1}\xi^{u})\vee f\beta)+ H_{\mu}( f^{n}\xi^{u}\,|\,(f^{n+1}\xi^{u})\vee \beta)\\
	&= H_{\mu}(\beta\,|\,(f^{n+1}\xi^{u})\vee f\beta)+ H_{\mu}( \xi^{u}\,|\,(f\xi^{u})\vee f^{-n}\beta).
	\end{aligned}$$
	
	By the third assumption on $\beta$, $f^{-n}\beta$ tends increasingly to the partition $\varepsilon$ into points. Note that $H_{\mu}( \xi^{u}\,|\,(f\xi^{u})\vee f^{-n}\beta)\leq H_{\mu}( \xi^{u}\,|\,f\xi^{u})<+\infty$, by Property 5.11 in \cite{Roh67}, the second term $H_{\mu}( \xi^{u}\,|\,(f\xi^{u})\vee f^{-n}\beta)$ above goes to $H_\mu(\xi^u\,|\,\varepsilon)=0$ as $n\to +\infty$. So it is sufficient to prove $H_{\mu}(\beta\,|\,(f^{n+1}\xi^{u})\vee f\beta)\to H(\beta\,|\,f\beta)=h_{\mu}(f,\beta).$ First note that $$H_{\mu}(\beta\,|\,(f^{n+1}\xi^{u})\vee f\beta)\leq H(\beta\,|\,f\beta).$$ 
	
 Write the conditional measures of $\mu$ w.r.t. $(f^{n+1}\xi^{u})\vee f\beta$ as $\{\mu^{n}_{x}\}_{x\in M}$ and the conditional measures w.r.t. $f\beta$ as $\{\mu_{x}\}_{x\in M}$. By definition, $$H_{\mu}(\beta\,|\,(f^{n+1}\xi^{u})\vee f\beta)=\int -\log\mu_{x}^{n}(\beta(x))\,d\,\mu(x),\quad H(\beta\,|\,f\beta)=\int-\log\mu_{x}(\beta(x))\,d\,\mu(x).$$ 
 
 	Let $$\Omega_{n}=\{x\,|\, f\beta(x)\subset f^{n+1}\xi^{u}(x)\}.$$ Since $\xi^{u}(x)$ contains an open neighborhood of $x$ in $W^{u}(x)$ w.r.t. the sub-manifold topology, by assumptions 2~and~3 on $\beta$, $\{\Omega_{n}\}$ is a non-decreasing sequence and $\mu(\Omega_{n})\to 1$ as $n\to +\infty$. For $x\in\Omega_{n}$, by definition, $(f^{n+1}\xi^{u})(x)\cap (f\beta)(x)=(f\beta)(x)$. Then one can show that this implies $$-\log\mu_{x}^{n}(\beta(x))=-\log\mu_{x}(\beta(x)),\quad \mu\,\, a.e.\,\,x\in\Omega_{n}.$$
 	
	Hence the non-negative functions $ \{-\log\mu_{(\cdot)}^{n}(\beta(\cdot))\}$ tend pointwise to $-\log\mu_{(\cdot)}(\beta(\cdot))$. By Fatou's Lemma, $$\lim_{n\to+\infty} H_{\mu}(\beta\,|\,(f^{n+1}\xi^{u})\vee f\beta)\geq H(\beta\,|\,f\beta).$$
\end{proof}

~\\
\section{Proof of Theorem \ref{main result}}
We prove the assertions in Theorem \ref{main result} one by one.
\subsection{$h_{i}(f,\mu)=\underline{h}_{i}^{K}(f,\mu)=\overline{h}_{i}^{K}(f,\mu)$}
We first prove $h_{i}(f,\mu)\leq\underline{h}_{i}^{K}(f,\mu)$.

By Proposition \ref{well defined hi} and by removing a set of zero measure from $R_{\mu}$ if necessary, we can assume that $$\lim_{\varepsilon\to 0}\limsup_{n\to+\infty}-\frac{1}{n}\log\mu^{i}_{x}(V^{i}(x,n,\varepsilon))=\lim_{\varepsilon\to 0}\liminf_{n\to+\infty}-\frac{1}{n}\log\mu^{i}_{x}(V^{i}(x,n,\varepsilon))=h_{i}(f,\mu),\quad\forall\,x\in R_{\mu}.$$	

We write $h_{i}(f,\mu)$ as $h_{i}$ for short.

For any $\eta,\varepsilon>0$, define $$\Delta^{\varepsilon}_{\eta}\triangleq \{x\in R_{\mu}\big|\,\liminf_{n\to+\infty}-\frac{1}{n}\log\mu^{i}_{x}(V^{i}(x,n,2\varepsilon))> h_{i}-\eta\}.$$  Then $\cup_{\varepsilon>0}\Delta^{\varepsilon}_{\eta}=R_{\mu}$. 

For $j\in\mathbb{N}$ and $p\in\Delta^{\varepsilon}_{\eta}$, define
$$\Delta^{\varepsilon}_{\eta}(p,j)\triangleq\{x\in\Delta_{\eta}^{\varepsilon}\big|\,\mu^{i}_{p}(V^{i}(x,n,2\varepsilon))\leq e^{-n(h_{i}-\eta)},\,\forall\,n\geq j\}.$$ By definition, $$\Delta^{\varepsilon}_{\eta}(p,j)\subset \Delta^{\varepsilon}_{\eta}(p,j+1),\quad\mu^{i}_{p}(\cup_{j}\Delta^{\varepsilon}_{\eta}(p,j))=\mu^{i}_{p}(\Delta^{\varepsilon}_{\eta}).$$ 

Fix any $\lambda>0$ and $p\in R_{\mu}$. Choose $\varepsilon$ small enough and $N$ large enough such that $$\mu^{i}_{p}(\Delta^{\varepsilon}_{\eta}(p,j))\geq 1-\frac{\lambda}{2},\quad\forall\,j\geq N.$$ For $n\in\mathbb{N}$, let $C_{n}\subset R_{\mu}$ be a subset such that $\# C_{n}=N_{\lambda}(\mu^{i}_{p},n,\varepsilon)$ and $\mu^{i}_{p}(\bigcup_{y\in C_{n}}V^{i}(y,n,\varepsilon))\geq\lambda$. Hence we have $$\mu^{i}_{p}(\Delta^{\varepsilon}_{\eta}(p,n)\cap (\bigcup_{y\in C_{n}}V^{i}(y,n,\varepsilon)))\geq \frac{\lambda}{2},\quad\forall\,n\geq N.$$

Let $A_{n}\subset C_{n}$ be such that for each $y\in C_{n}$, we have $V^{i}(y,n,\varepsilon)\cap \Delta^{\varepsilon}_{\eta}(p,n)\neq \emptyset$. For $y\in A_{n}$, we fix any $\widetilde{y}\in V^{i}(y,n,\varepsilon)\cap \Delta^{\varepsilon}_{\eta}(p,n)$. Then we have $$V^{i}(y,n,\varepsilon)\subset V^{i}(\widetilde{y},n,2\varepsilon).$$ Hence for $n\geq N$, $$\frac{\lambda}{2}\leq \mu^{i}_{p}(\bigcup_{y\in A_{n}}V^{i}(\widetilde{y},n,2\varepsilon))\leq N_{\lambda}(\mu^{i}_{p},n,\varepsilon)\times\sup_{x\in \Delta^{\varepsilon}_{\eta}(p,n)}\mu^{i}_{p}(V^{i}(x,n,2\varepsilon))\leq N_{\lambda}(\mu^{i}_{p},n,\varepsilon)\times e^{-n(h_{i}-\eta)}.$$

Therefore for any $\varepsilon$ small enough(depending on $\eta$) and $n$ large enough, $$N_{\lambda}(\mu^{i}_{p},n,\varepsilon)\geq\frac{\lambda}{2}\cdot e^{n(h_{i}-\eta)}.$$
Then by the arbitrariness of $\eta$ and $\lambda$, we get $$h_{i}(f,\mu)\leq \inf_{\lambda}\lim_{\varepsilon\to 0}\liminf_{n\to+\infty}\frac{1}{n}\log N_{\lambda}(\mu^{i}_{p},n,\varepsilon)=\underline{h}_{i}^{K}(f,\mu).$$
\\[0.5cm]
Next we prove $h_{i}(f,\mu)\geq\overline{h}_{i}^{K}(f,\mu)$. The arguments are similar as above.

For any $\eta,\varepsilon>0$, define $$\Omega^{\varepsilon}_{\eta}\triangleq \{x\in R_{\mu}\big|\,\limsup_{n\to+\infty}-\frac{1}{n}\log\mu^{i}_{x}(V^{i}(x,n,\frac{\varepsilon}{2}))<h_{i}+\eta\}.$$  Then $\cup_{\varepsilon>0}\Delta^{\varepsilon}_{\eta}=R_{\mu}$. 

For $j\in\mathbb{N}$ and $p\in\Delta^{\varepsilon}_{\eta}$, define
$$\Omega^{\varepsilon}_{\eta}(p,j)\triangleq\{x\in\Delta_{\eta}^{\varepsilon}\big|\,\mu^{i}_{p}(V^{i}(x,n,\frac{\varepsilon}{2}))\geq e^{-n(h_{i}+\eta)},\,\forall\,n\geq j\}.$$ By definition, $$\Omega^{\varepsilon}_{\eta}(p,j)\subset \Omega^{\varepsilon}_{\eta}(p,j+1),\quad\mu^{i}_{p}(\cup_{j}\Omega^{\varepsilon}_{\eta}(p,j))=\mu^{i}_{p}(\Omega^{\varepsilon}_{\eta}).$$ 

Fix any $\lambda>0$ and $p\in R_{\mu}$. Choose $\varepsilon$ small enough and $N$ large enough such that $$\mu^{i}_{p}(\Omega^{\varepsilon}_{\eta}(p,j))\geq\lambda,\quad\forall\,j\geq N.$$

For $n\in\mathbb{N}$, let $F_{n}\subset \Omega^{\varepsilon}_{\eta}(p,n)$ be a maximal $(n,\varepsilon)$ separated set of $\Omega^{\varepsilon}_{\eta}(p,n)\cap \xi^{i}(p).$ Then $\{V^{i}(y,n,\varepsilon)\}_{y\in F_{n}}$ covers $\Omega^{\varepsilon}_{\eta}(p,n)\cap \xi^{i}(p).$ Hence $\#F_{n}\geq N_{\lambda}(\mu^{i}_{p},n,\varepsilon)$. And we also have $$y_{1},y_{2}\in F_{n},\, y_{1}\neq y_{2}\Longrightarrow V^{i}(y_{1},n,\frac{\varepsilon}{2})\bigcap V^{i}(y_{2},n,\frac{\varepsilon}{2})=\emptyset.$$ Hence for $n\geq N$, $$N_{\lambda}(\mu^{i}_{p},n,\varepsilon)\leq\# F_{n}\leq\frac{1}{\sup_{x\in\Omega^{\varepsilon}_{\eta}(p,n)}\mu^{i}_{p}(V^{i}(x,n,\frac{\varepsilon}{2}))}\leq e^{n(h_{i}+\eta)}.$$

Then by the arbitrariness of $\eta$ and $\lambda$, we get $$h_{i}(f,\mu)\geq \inf_{\lambda}\lim_{\varepsilon\to 0}\limsup_{n\to+\infty}\frac{1}{n}\log N_{\lambda}(\mu^{i}_{p},n,\varepsilon)=\overline{h}_{i}^{K}(f,\mu).$$

\subsection{$h_{i}(f,\mu)\leq \underline{h}^{i}_{\top}(f,\mu)$}

Since $\xi^{i}$ is a partition subordinate to $W^{i}$, for any $\rho>0$, we assume for any $x\in R_{\mu}$, $\mu^{i}_{x}(W^{i}_{\rho}(x))>0.$

For $\rho,\varepsilon,\eta>0, n,j\in\mathbb{N}$ and $p\in R_{\mu}$, let $F_{\eta}^{\varepsilon}(p,j,n)$ be a  $(n,\varepsilon)$-separated subset of $\Delta^{\varepsilon}_{\eta}(p,j)\cap W^{i}_{\rho}(p)$ with maximum cardinality. It is a cover, hence we have 
$$\begin{aligned}
\mu^{i}_{p}(\Delta^{\varepsilon}_{\eta}(p,j)\cap W^{i}_{\rho}(p))
&\leq \mu_{p}^{i}(\bigcup_{x\in F_{\eta}^{\varepsilon}(p,j,n)}V^{i}(x,n,\varepsilon))\\
&\leq \# F_{\eta}^{\varepsilon}(p,j,n)\times\sup_{x\in\Delta^{\varepsilon}_{\eta}(p,j)}\mu^{i}_{p}(V^{i}(x,n,\varepsilon))\\
&\leq \# F_{\eta}^{\varepsilon}(p,j,n)\times e^{-n(h_{i}-\eta)},\quad\forall\,j,\,\forall\, n\geq j.
\end{aligned}$$

Hence $$\# F_{\eta}^{\varepsilon}(p,j,n)\geq\frac{\mu^{i}_{p}(\Delta^{\varepsilon}_{\eta}(p,j)\cap W^{i}_{\rho}(p))}{e^{n(h_{i}-\eta)}},\quad\forall\,j,\,\forall\, n\geq j.$$
Choose $\varepsilon$ small enough (depending on $\eta$) and $j$ large enough such that $\mu^{i}_{p}(\Delta^{\varepsilon}_{\eta}(p,j)\cap W^{i}_{\rho}(x))>0$. Then taking $\liminf_{n\to+\infty}\frac{1}{n}\log$ on both sides, we have $$\liminf_{n\to+\infty}\frac{1}{n}\log \#E(n,\varepsilon,W^{i}_{\rho}(x))\geq\liminf_{n\to+\infty}\frac{1}{n}\log \# F_{\eta}^{\varepsilon}(p,j,n)\geq h_{i}-\eta.$$
Since $\eta$ and $\rho$ are arbitrary, we get $ h_{i}(f,\mu)\leq\underline{h}^{i}_{\top}(f,\mu)$.

\subsection{$h_{i}(f,\mu)\leq \vl_{i}(f,\mu)$}
Applying Lemma  \ref{pesin theory} for any $\varepsilon>0$, we obtain  an increasing sequence of measurable sets $\{\Lambda^{\varepsilon}_{k}\subset R_{\mu}\}$.

Let us first note that, for any $k,n\in\mathbb{N}$, any $x\in\Lambda^{\varepsilon}_{k}$, any $\tau\leq r_{k}e^{-n\varepsilon}$ and any $y$ with $f^{n}(y)\in W^{i}_{\tau}(f^{n}(x))$, $$d(f^{n-j}(x),f^{n-j}(y))\leq r_{k+n}^{-1}e^{-j(\lambda_{i}-\varepsilon)}d(f^{n}(x),f^{n}(y)),\,\forall\,0\leq j\leq n.$$ Hence for any $k,n\in\mathbb{N}$, any $x\in\Lambda^{\varepsilon}_{k}$ and any $\tau\leq r_{k}e^{-n\varepsilon}$, $f^{n}(V^{i}(x,n,\tau))$ contains an  $i$-th local sub-manifold $W^{i}_{r_{k+n}\tau}(f^{n}(x))$.

%where $B^{i}(x,\rho)\triangleq W^{i}_{\rho}(x)\cap B(x,\rho)$.
Since the function $\mu^{i}_{p}(V^{i}(p,n,\tau))$ is non-decreasing w.r.t. $\tau$, for any sequence $\{\tau_{n}\}$ with $\tau_{n}\to 0$, we have for $p\in R_{\mu}$,

$$\liminf_{n\to+\infty}-\frac{1}{n}\log\mu^{i}_{p}(V^{i}(p,n,\tau_{n}))\geq \liminf_{n\to+\infty}-\frac{1}{n}\log\mu^{i}_{p}(V^{i}(p,n,\tau)),\quad\forall\,\tau>0.$$

Hence in particular, for $k\in\mathbb{N},\,p\in\Lambda_{k}^{\varepsilon}$ and $ x\in\Lambda^{\varepsilon}_{k}\cap \xi^{i}(p),$ $$\liminf_{n\to+\infty}-\frac{1}{n}\log\mu^{i}_{p}(V^{i}(x,n,r_{k}e^{-n\varepsilon}))\geq h_{i}.$$

For any $j\in\mathbb{N},\,p\in\Lambda_{k}^{\varepsilon}$ and $\rho>0$ with $W^{i}_{\rho}(p)\subset\xi^{i}(p)$,  define $$\Lambda^{\varepsilon,p}_{k,j}=\{x\in\Lambda^{\varepsilon}_{k}\cap W^{i}_{\rho}(p)\,\big|\,\mu^{i}_{p}(V^{i}(x,n,r_{k}e^{-n\varepsilon}))\leq e^{-n(h_{i}-\varepsilon)},\,\forall\,n\geq j\}.$$ By definition, $$\Lambda^{\varepsilon,p}_{k,j}\subset \Lambda^{\varepsilon,p}_{k,j+1},\quad\mu^{i}_{p}(\bigcup_{j}\Lambda^{\varepsilon,p}_{k,j})=\mu^{i}_{p}(\Lambda^{\varepsilon}_{k}\cap W^{i}_{\rho}(p)).$$

Let $F_{n,j,k}^{\varepsilon,p}$ be a  $(n,r_{k}e^{-n\varepsilon})$-separated subset of $\Lambda^{\varepsilon,p}_{k,j}$ with maximum cardinality. Then we have
$$\begin{aligned}
\mu^{i}_{p}(\Lambda^{\varepsilon,p}_{k,j})
&\leq \mu^{i}_{p}(\bigcup_{x\in F_{n,j,k}^{\varepsilon,p}}V^{i}(x,n,r_{k}e^{-n\varepsilon}))\\
&\leq \# F_{n,j,k}^{\varepsilon,p}\times\sup_{x\in\Lambda^{\varepsilon,p}_{k,j}}\mu^{i}_{p}(V^{i}(x,n,r_{k}e^{-n\varepsilon})).
\end{aligned}\eqno(*)$$

Note that for any $x\in F_{n,j,k}^{\varepsilon,p}\subset W^{i}_{\rho}(p)$, $V^{i}(x,n,r_{k}e^{-n\varepsilon})\subset W^{i}_{2\rho}(p)$ for all $n$ such that $r_{k}e^{-n\varepsilon}<\rho$. Since the sets $\{f^{n}(V^{i}(x,n,\frac{1}{2}r_{k}e^{-n\varepsilon}))\}_{x\in F_{n,j,k}^{\varepsilon,p}}$ are mutually disjoint, for all large $n$, we have $$\vol(f^{n}(W^{i}_{2\rho}(x)))\geq\sum_{x\in F_{n,j,k}^{\varepsilon,p}}\vol(f^{n}(V^{i}(x,n,\frac{1}{2}r_{k}e^{-n\varepsilon}))).$$ Recall that each $f^{n}(V^{i}(x,n,\frac{1}{2}r_{k}e^{-n\varepsilon}))$ contains an  $i$-th local unstable manifold $W^{i}_{\frac{1}{2}r_{k+n}r_{k}e^{-n\varepsilon}}(f^{n}(x))$. Thus $$\vol(f^{n}(W^{i}_{2\rho}(x)))\geq\# F_{n,j,k}^{\varepsilon,p}\times K\times(\frac{1}{2}r_{k+n}r_{k}e^{-n\varepsilon})^{\sum_{l=1}^{i}\dim E^{l}}.\eqno(**)$$ where $K$ is the constant from Lemma \ref{pesin theory}.

Combining $(\ast)$ and $(\ast\ast)$, for all large $n$,
$$\begin{aligned}
\vol(f^{n}(W^{i}_{2\rho}(x)))
&\geq\frac{\mu^{i}_{p}(\Lambda^{\varepsilon,p}_{k,j})}{\sup_{x\in\Lambda^{\varepsilon,p}_{k,j}}\mu^{i}_{p}(V^{i}(x,n,r_{k}e^{-n\varepsilon}))}\times K\times(\frac{1}{2}r_{k+n}r_{k}e^{-n\varepsilon})^{\sum_{l=1}^{i}\dim E^{l}}\\
&\geq\frac{\mu^{i}_{p}(\Lambda^{\varepsilon,p}_{k,j})}{e^{-n(h_{i}-\varepsilon)}}\times K\times(\frac{1}{2}r_{k+n}r_{k}e^{-n\varepsilon})^{\sum_{l=1}^{i}\dim E^{l}},\,\,\,\,\forall\,k\in\mathbb{N},\,\forall\,p\in\Lambda_{k}^{\varepsilon},\,\forall\, j\in\mathbb{N}.
\end{aligned}$$

Since $\mu^{i}_{p}(W^{i}_{\rho}(p))>0$, we choose $k,j$ large enough such that $\mu^{i}_{p}(\Lambda^{\varepsilon,p}_{k,j})>0$. Taking $\liminf_{n\to+\infty}\frac{1}{n}\log$ on both sides, we have $$\liminf_{n\to+\infty}\frac{1}{n}\log \vol(f^{n}(W^{i}_{2\rho}(p)))\geq h_{i}-\varepsilon-2\varepsilon\sum_{l=1}^{i}\dim E^{l},\,\,\,\,\forall\,k\in\mathbb{N},\,\forall\,p\in\Lambda_{k}^{\varepsilon}.$$ Hence we have $$\liminf_{n\to+\infty}\frac{1}{n}\log \vol(f^{n}((W^{i}_{2\rho}(p)))\geq h_{i}-\varepsilon-2\varepsilon\sum_{l=1}^{i}\dim E^{l},\,\,\,\,\forall\,p\in R_{\mu}.$$

By the arbitrariness of $\varepsilon$ and $\rho$, we get the conclusion.

\section{Proof of Theorem \ref{upper bound mixture C1+ setting}}
We now explain how to deduce our Theorem \ref{upper bound mixture C1+ setting} based on the arguments of Ledrappier and Young in  \cite{LeY85I}.

\begin{proof}
	Let $\xi^{u}$ be any increasing measurable partition subordinate to $W^{u}$. By Proposition \ref{well defined hi} (see Remark \ref{entropy for increasing partition}), $h_{u}(f,\mu)=H_{\mu}(\xi^{u}|f\xi^{u})=h_{\mu}(f,\xi^{u})$. Hence $h_{u}(f,\mu)\leq h(f,\mu)$. So it is sufficient to prove $h_{u}(f,\mu)\geq h(f,\mu)$. 
	
	In the following argument, some properties only hold for $\mu$-a.e. $x$. But without loss of generality, we assume these properties hold for any $x\in R_{\mu}$.

	For $\varepsilon>0$ and $x\in \Lambda_{k}^{\varepsilon}$, let $S^{cu}(x)$ be the set in Lemma \ref{hyperbolicity implies points always in chart are in unstable manifold}.
	By Lemma 2.4.2 in  \cite{LeY85I}, there is a measurable partition $\xi$ with $H_{\mu}(\xi)<+\infty$ such that $\xi^{+}(x)\subset S^{cu}(x),\,x\in R_{\mu}$ where $\xi^{+}=\vee_{n=0}^{+\infty}f^{n}\xi$. Since $H_{\mu}(\xi)<+\infty$, we can assume $h_{\mu}(f,\xi)\geq h(f,\mu)-\varepsilon$.

	By Lemma \ref{hyperbolicity implies points always in chart are in unstable manifold}, $\xi^{+}(x)\subset W^{u}(x),\,x\in R_{\mu}$. 
	
	We note the following facts:
	\begin{itemize}
		\item By Lemma 3.2.1 in  \cite{LeY85I}, we have $$h_{u}(f,\mu)=H_{\mu}(\xi^{u}|f\xi^{u})=h_{\mu}(f,\xi^{u}\vee \xi^{+}).$$
		\item Since $\xi^{+}$ is increasing, $\xi^{+}(x)\subset W^{u}(x)$ and $\diam(f^{-n}(\xi^{+}(x)))\to 
		0$ for $\mu$-a.e. $x$, Lemma~\ref{extension of the relations between subordinate partitions} yields $$h_{\mu}(f,\xi^{u}\vee \xi^{+})= h_{\mu}(f,\xi^{+}).$$
		\item Since $\xi^{+}$ is increasing, 
		$$h_{\mu}(f,\xi^{+})=H_{\mu}(\xi^{+}|f(\xi^{+}))=H_{\mu}(\xi\vee f(\xi^{+})|f(\xi^{+}))=H_{\mu}(\xi|f(\xi^{+}))=h_{\mu}(f,\xi).$$
	\end{itemize}
	
	Hence we have $$h_{u}(f,\mu)=h(f,\xi)\geq h(f,\mu)-\varepsilon.$$
	
	Since $\varepsilon$ is arbitrarily small, we get the conclusion.
	
\end{proof}

\begin{Acknowledgements}
	I would like to thank my advisors, J\'{e}r\^{o}me Buzzi and Dawei Yang, for their great patience, encouragement and math guidance. I also thank Prof. Sylvain Crovisier and Jinhua Zhang for helpful discussions. 
\end{Acknowledgements}

~\\
\begin{tabular}{l l l}

	\emph{\normalsize Yuntao ZANG}
	\medskip\\
	
	\small Laboratoire de Math\'ematiques d'Orsay
	\\
	\small CNRS - Universit\'e Paris-Sud
	\\
	\small Orsay 91405, France
	\\
		\small School of Mathematical Sciences
	\\
	\small Soochow University
	\\
	\small Suzhou, 215006, P.R. China
	\\
	
	\small \texttt{yuntao.zang@math.u-psud.fr }
	
\end{tabular}

\end{document}